\documentclass[10pt]{amsart}

\usepackage{amsmath,amsfonts,latexsym,amssymb,mathrsfs,setspace,enumerate,color,cite,graphicx,epsf}

\newtheorem{thm}{Theorem}
\newtheorem{lem}[thm]{Lemma}

\numberwithin{equation}{section}
\numberwithin{thm}{section}

\newtheorem{conj}[thm]{Conjecture}
\newtheorem{question}[thm]{Question}

\newcommand{\rat}{\mathbb Q}

\newcommand{\alg}{\overline\rat}
\newcommand{\algt}{\alg^{\times}}

\newcommand{\intg}{\mathbb Z}
\newcommand{\nat}{\mathbb N}

\newcommand{\lcm}{\mathrm{lcm}}

\newcommand{\tor}{\mathrm{Tor}}

\newcommand{\gal}{\mathrm{Gal}}
\newcommand{\norm}{\mathrm{Norm}}
\newcommand{\rad}{\mathrm{Rad}}

\def\Z{\mathbb{Z}}
\def\N{\mathbb{N}}
\def\Q{\mathbb{Q}}
\def\al{\alpha}

\newcommand{\comment}[1]{}

\title[The $t$-metric Mahler measure]{The $t$-metric Mahler measures of surds of rational numbers}

\author[J. Jankauskas \and C.L. Samuels]{Jonas Jankauskas \and Charles L. Samuels}
\address{Vilnius University, Department of Probability Theory and Number Theory, Faculty of Mathematics and Informatics, Naugarduko 24, LT-03225 Vilnius, Lithuania}
\email{jonas.jankauskas@gmail.com}
\thanks{The first author was supported by the Lithuanian Research Council (student research support project) during his visit to the IRMACS Centre, Simon Fraser University.}
\address{Simon Fraser University, Department of Mathematics, 8888 University Drive, Burnaby, BC V5A 1S6, Canada \newline \indent The University of British Columbia, 
Department of Mathematics, 1984 Mathematics Road, Vancouver, BC V6T 1Z2, Canada}
\email{csamuels@math.ubc.ca}
\thanks{The second author was supported by NSERC of Canada}


\begin{document}

\begin{abstract}
A. Dubickas and C. Smyth introduced the metric Mahler measure
\begin{equation*}
  M_1(\al) = \inf\left\{\sum_{n=1}^N M(\al_n): N \in \N, \al_1 \cdots \al_N = \al\right\},
\end{equation*}
where $M(\al)$ denotes the usual (logarithmic) Mahler measure of $\al \in \alg$. This definition extends in a natural way
to the $t$-metric Mahler measure by replacing the sum with the usual $L_t$ norm of the vector $(M(\al_1), \dots, M(\al_N))$ for any $t\geq 1$. 
For $\al \in \rat$, we prove that the infimum in $M_t(\al)$ may be attained using only rational points, establishing an earlier conjecture of the second author.
We show that the natural analogue of this result fails for general $\alpha\in\alg$ by giving an infinite family of quadratic counterexamples.
As part of this construction, we provide an explicit formula to compute $M_t(D^{1/k})$ for a squarefree $D \in \N$.
\end{abstract}

\maketitle

\section{Introduction}

Let $f$ be a polynomial with complex coefficients given by
\begin{equation*}
	f(z) = a\cdot\prod_{n=1}^N(z-\alpha_n).
\end{equation*}
Recall that the {\it (logarithmic) Mahler measure} $M$ of $f$ is defined by
\begin{equation*}
	M(f) = \log|a|+\sum_{n=1}^N\log^+|\alpha_n|.
\end{equation*}
If $\alpha$ is a non-zero algebraic number, the {\it (logarithmic) Mahler measure $M(\alpha)$ of $\alpha$} is defined as the Mahler measure of the minimal 
polynomial of $\alpha$ over $\intg$.

It is a consequence of a theorem of Kronecker that $M(\alpha) = 0$ if and only if $\alpha$ is a root of unity.  In a famous 1933 paper, D.H. Lehmer \cite{Lehmer} 
asked whether there exists a constant $c>0$ such that $M(\alpha) \geq c$ in all other cases.  He could find no algebraic number with Mahler measure smaller than that of
\begin{equation*}
  \ell(x) = x^{10}+x^9-x^7-x^6-x^5-x^4-x^3+x+1,
\end{equation*}
which is approximately $0.16\ldots$.  Although the best known general lower bound is
\begin{equation*}
	M(\alpha) \gg \left(\frac{\log\log\deg\alpha}{\log\deg\alpha}\right)^3,
\end{equation*}
due to Dobrowolski \cite{Dobrowolski},
uniform lower bounds have been established in many special cases (see \cite{BDM, Schinzel, Smyth}, for instance).  Furthermore, numerical evidence
provided, for example, in \cite{Moss, MossWeb, MPV, MRW} suggests there exists such a constant $c$.  

\begin{conj}[Lehmer's conjecture] \label{LehmerConjecture}
	There exists a real number $c > 0$ such that if $\alpha\in\algt$ is not a root of unity then $M(\alpha) \geq c$.
\end{conj}

For an algebraic number $\alpha$, Dubickas and Smyth \cite{DubSmyth2} introduced the {\it metric Mahler measure} $M_1(\alpha)$ by
\begin{equation} \label{DubSmythMetric}
  M_1(\alpha) = \inf\left\{\sum_{n=1}^NM(\alpha_n): N\in\nat,\ \alpha_n\in\algt,\ \alpha=\prod_{n=1}^N\alpha_n\right\}.
\end{equation}
Here, the infimum is taken over all ways to write $\alpha$ as a product of algebraic numbers.  The advantage of $M_1$ over $M$ is that it satisfies the 
triangle inequality
\begin{equation*}
  M_1(\alpha\beta) \leq M_1(\alpha) + M_1(\beta)
\end{equation*}
for all algebraic numbers $\alpha$ and $\beta$.  In view of this observation, $M_1$ is well-defined on the quotient group $G = \algt/\tor(\algt)$,
and the map $(\alpha,\beta) \mapsto M_1(\alpha\beta^{-1})$ defines a metric on $G$.  This metric induces the discrete topology if and only if Lehmer's conjecture is true. 

The metric Mahler measure $M_1$ is only a special case of the {\it $t$-metric Mahler measures}, which are defined for $t \geq 1$ by
\begin{equation*}
  M_t(\alpha) = \inf\left\{\left(\sum_{n=1}^NM(\alpha_n)^t\right)^{1/t}:N\in\nat,\ \alpha_n\in\algt,\  \alpha=\prod_{n=1}^N\alpha_n\right\}.
\end{equation*}
In addition, the {\it $\infty$-metric Mahler measure} of $\alpha$ is defined by
\begin{equation*}
  M_\infty(\alpha) = \inf\left\{\max_{1\leq n\leq N}\{M(\alpha_n)\}: N\in\nat,\ \alpha_n\in\algt,\ \alpha = \prod_{n=1}^N\alpha_n\right\}.
\end{equation*}
The $t$-metric Mahler measures were introduced and studied in \cite{SamuelstMetric1, SamuelstMetric2}. 
It follows from the results of \cite{SamuelstMetric1} that these functions have analogues of the triangle inequality
\begin{equation*}
  M_t(\alpha\beta)^t \leq M_t(\alpha)^t + M_t(\beta)^t\quad\mathrm{and}\quad M_\infty(\alpha\beta)\leq\max\{M_\infty(\alpha),M_\infty(\beta)\}
\end{equation*}
Hence, the map $(\alpha,\beta) \mapsto M_t(\alpha\beta^{-1})$ defines a metric on $G$ that induces the discrete topology if and only if Lehmer's conjecture is true.

If $t\in [1,\infty]$ and $\alpha\in\alg$, we say that {\it the infimum in $M_t(\alpha)$ is attained by $\alpha_1,\ldots,\alpha_N$} if we have that
\begin{equation*}
	\alpha = \alpha_1\cdots\alpha_N\quad\mathrm{and}\quad 
		M_t(\alpha) =  \begin{cases} \left(\sum_{n=1}^N M(\alpha_n)^t\right)^{1/t} & \mathrm{if}\ t<\infty \\
								\max_{1\leq n\leq N} \{M(\alpha_n)\} & \mathrm{if}\ t = \infty. \end{cases}
\end{equation*}
If $S$ is any subset of $\alg$, we say {\it the infimum in $M_t(\alpha)$ is attained in $S$} if there exist points $\alpha_1,\ldots,\alpha_N\in S$ that attain the infimum in $M_t(\alpha)$. 

It is not immediately obvious that $M_t(\al)$ is attained for all values of $\alpha$ and $t$.  Dubickas and Smyth \cite{DubSmyth2} conjectured that the infimum in 
$M_1(\alpha)$ is always attained a fact later proved by the second author \cite{SamuelsInfimum}. More specifically, if $K_\alpha$ is the Galois closure of $\rat(\alpha)$ over $\rat$ and
\begin{equation*}
  \rad(K_\alpha) = \{\gamma\in\alg:\gamma^n\in K_\alpha\mbox{ for some } n\in\nat\},
\end{equation*}
then the infimum in $M_1(\alpha)$ is attained in $\rad(K_\alpha)$. Using the same method, this result was generalized for all $t$-metric Mahler measures in \cite{SamuelstMetric1}. 
That is, for every $t \geq 1$, the infimum in $M_t(\alpha)$ is attained in $\rad(K_\alpha)$.

It is natural to ask if these results can be improved, having a smaller set $S$ in place of $\rad(K_\alpha)$.  In particular, for each $\alpha\in\alg$, we would like to identify a set $S_\alpha$
whose points generate a finite extension of $\rat$ and the infimum in $M_t(\alpha)$ is attained in $S_\alpha$ for all $t$.  This problem is of considerable importance if we 
hope to compute exact values of $M_t(\alpha)$.  
For example, Conjecture 2.1 of \cite{SamuelstMetric2} predicts that, if $\alpha$ is rational, then the infimum in $M_t(\alpha)$ is attained in $\rat$.  With this assumption, it is
possible to graph some examples of the function $t\mapsto M_t(\alpha)$ where $\alpha\in \rat$.

It follows from \cite{DubSmyth2} and \cite{FiliSamuels} that Conjecture 2.1 of \cite{SamuelstMetric1} holds for $t = 1$ and $t=\infty$.  Unfortunately, these methods
seem genuinely distinct and cannot be easily generalized to handle all values of $t$ and $\al$.   As our first result, we prove this conjecture for all $t \geq 1$.

\begin{thm} \label{RationalAttained}
  If $\alpha$ is a non-zero rational number and $t\in[1,\infty]$ then the infimum in $M_t(\alpha)$ is attained in $\rat$.
\end{thm}

Our next question is whether Theorem \ref{RationalAttained} can be extended to arbitrary algebraic numbers $\al$.
In view of Theorem \ref{RationalAttained}, one might suspect that the infimum in $M_t(\alpha)$ is always attained in $K_\alpha$.
This turns out to be false, however, as we are able to produce an infinite family of quadratic counterexamples.
More specifically, if $D$ is a square-free positive integer, we show precisely when $M_t(\sqrt D)$ is attained in $K_{\sqrt D} = \rat(\sqrt D)$.

\begin{thm} \label{QuadraticNotAttained}
	Suppose that $p_1,\ldots, p_L$ are distinct primes written in decreasing order, $D = p_1\cdots p_L$, and $t\in (1,\infty]$.
	The infimum in $M_t(\sqrt D)$ is attained in $\rat(\sqrt D)$ if and only if $D < p_1^2$.  In this situation, the infimum is attained by points
	\begin{equation*}
		\sqrt{\frac{p_1}{p_2\cdots p_L}}, p_2,\ldots, p_L \in\rat(\sqrt D),
	\end{equation*}
	and we have that
	\begin{equation*}
		M_t(\sqrt D) = \begin{cases} \left(\sum_{\ell = 1}^L (\log p_\ell)^t\right)^{1/t} & \mathrm{if\ } t\in (1,\infty) \\
					\log p_1 & \mathrm{if\ } t = \infty. \end{cases}
	\end{equation*}
\end{thm}

Theorem \ref{QuadraticNotAttained} enables the construction of infinitely many integers $D$ such that $M_t(\sqrt D)$
is not attained in $K_{\sqrt D} = \rat(\sqrt D)$ for any $t>1$.  Theorem \ref{QuadraticValue} below gives a set of points that attain the infimum in $M_t(\al)$ for algebraic 
numbers $\al = D^{1/k}$, where $D > 0$ is a square-free integer.

\begin{thm} \label{QuadraticValue}
	If $p_1,\ldots, p_L$ are distinct primes, $D = p_1\cdots p_L$, and $t\in [1,\infty]$, then
	the infimum in $M_t(D^{1/k})$ is attained by $p_1^{1/k},\ldots,p_L^{1/k}$ and
	\begin{equation*}
		M_t(D^{1/k}) = \begin{cases} \left(\sum_{\ell = 1}^L (\log p_\ell)^t\right)^{1/t} & \mathrm{if\ } t\in [1,\infty) \\
					\max_{1\leq \ell\leq L}\{\log p_\ell\} & \mathrm{if\ } t = \infty. \end{cases}
	\end{equation*}
\end{thm}

As an example, for $D=30 = 2\cdot 3\cdot 5$, Theorem \ref{QuadraticValue} asserts that $M_t(\sqrt{30})$ is attained by $\sqrt 2, \sqrt 3, \sqrt 5$, and
\begin{equation*}
	M_t(\sqrt{30})^t = (\log 5)^t + (\log 3)^t + (\log 2)^t.
\end{equation*}
While it is obvious that $\sqrt 2,\sqrt 3,\sqrt 5\not\in \rat(\sqrt {30})$, the infimum in $M_t(\sqrt {30})$ might be attained by some distinct set of points
in $\rat(\sqrt{30})$.  Theorem \ref{QuadraticNotAttained} excludes any such possibilities.

If we take $D = 42 = 2\cdot 3\cdot 7$, Theorem \ref{QuadraticValue} establishes that $M_t(\sqrt{42})$ is attained by $\sqrt 2,\sqrt 3,\sqrt 7$, and
\begin{equation*}
	M_t(\sqrt{42})^t = (\log 7)^t + (\log 3)^t + (\log 2)^t.
\end{equation*}
Nonetheless, Theorem \ref{QuadraticNotAttained} identifies the slightly more subtle points $\sqrt{7/6},3,2\in\rat(\sqrt{42})$ that also attain the infimum in $M_t(\sqrt{42})$. 
In this example, we note that the infimum is not attained by a unique set.

At first glance, one might think that the infimum in $M_t(\sqrt{D})$ can be attained only by rational numbers and their square roots. 
This intuition is misleading, however, as we see in the following example.   Let $t = \infty$ and take $D=21=7\cdot 3$. 
We know from Theorem \ref{QuadraticNotAttained} that the infimum in $M_t(\sqrt{21})$ is attained by the points $\sqrt{7/3}, 3 \in \rat(\sqrt{21})$ and
\begin{equation*}
  M_t(\sqrt{21})^t = (\log 7)^t + (\log 3)^t.
\end{equation*}
Now consider
\begin{equation} \label{BetterInf}
  \sqrt{21} = (-1) \cdot \left(\frac{7+\sqrt{21}}{2}\right) \cdot \left(\frac{3-\sqrt{21}}{2}\right),
\end{equation}
and we verify easily that
\begin{equation*}
  M\left(\frac{7+\sqrt{21}}{2}\right) = \log 7\quad\mathrm{and}\quad M\left(\frac{3-\sqrt{21}}{2}\right) < \log 7.
\end{equation*}
In other words, $M_\infty(\sqrt{21})$ is attained by the points on the right hand side of \eqref{BetterInf} and these points belong to $\Q(\sqrt{21})$.  It is important to note that $M\left((3-\sqrt{21})/2\right) > \log 3$, so these points cannot be used to attain the infimum in $M_t(\sqrt{21})$ for other values of $t$. Nonetheless, this example illustrates that the infimum in $M_t(\sqrt D)$ may be attained by using distinct non-trivial sets of points contained in $\Q(\sqrt{D})$.

We would like to conclude with  the following question.

\begin{question}
  Is the infimum in $M_t(\al)$ always attained by points $\al_1, \dots, \al_N$ such that $[\rat(\al_n) :\rat] \leq [\rat(\al) : \rat]$ for all $n$?
\end{question}

According to Theorem \ref{QuadraticValue}, the answer is 'yes' when $\alpha$ is a surd, although we know of little other evidence.


\section{The rational case}

Recall that the {\it (logarithmic) Weil height} of an algebraic number $\alpha$ is given by
\begin{equation*}
  h(\alpha) = \frac{M(\alpha)}{\deg\alpha}.
\end{equation*}
It is well-known that if $\zeta$ is a root of unity, then $h(\alpha) = h(\zeta\alpha)$ so that $h$ is well-defined on our quotient group $G$.
Furthermore, if $n$ is an integer, then we have that $h(\alpha^n) = |n|\cdot h(\alpha)$.  Also recall that a surd is an algebraic number $\alpha$ such that $\alpha^n\in\rat$
for some positive integer $n$.

Suppose now that $F$ is any number field containing the algebraic number $\alpha$.  Further assume that $K$ is an extension of $F$ which is Galois over $\rat$.  We set
\begin{equation*}
  G = \gal(K/\rat)\quad\mathrm{and}\quad H = \gal(K/F),
\end{equation*}
and let $S$ be a set of left coset representatives of $H$ in $G$.  Recall that the {\it norm of $\alpha$ from $F$ to $\rat$} is given by
\begin{equation*}
  \norm_{F/\rat}(\alpha) = \prod_{\sigma\in S}\sigma(\alpha).
\end{equation*}
It follows from standard Galois Theory that $\norm_{F/\rat}$ is a homomorphism from $F$ to $\rat$ which does not depend on the choice of $K$ or $S$. 
In addition, if $E$ is any extension of $F$, then it is easily verified that
\begin{equation} \label{NormDescend}
  \norm_{E/\rat}(\alpha) = \norm_{F/\rat}(\alpha)^{[E:F]}.
\end{equation}
We begin our proof of Theorem \ref{RationalAttained} with a lemma that relates the Mahler measure of a surd to the Mahler measure
of its norm.

\begin{lem} \label{NormConversion}
  If $\gamma$ is a surd then $M(\gamma) = M\left(\norm_{\rat(\gamma)/\rat}(\gamma)\right)$.
\end{lem}
\begin{proof}
  Since $\gamma$ is a surd, its conjugates over $\rat$ are given by
  \begin{equation*}
    \{\zeta_1\gamma,\zeta_2\gamma,\ldots,\zeta_M\gamma\}
  \end{equation*}
  where $M=\deg \gamma$ and $\zeta_m$ are roots of unity.  It now follows that
  \begin{equation*}
    \gamma^M\prod_{m=1}^M\zeta_m = \norm_{\rat(\gamma)/\rat}(\gamma) \in\rat.
  \end{equation*}
  Since $\norm_{\rat(\gamma)/\rat}(\gamma)$ is clearly a rational number, we have that
  \begin{equation*}
    M\left(\norm_{\rat(\gamma)/\rat}(\gamma)\right) = h\left(\gamma^M\prod_{m=1}^M\zeta_m\right) = M\cdot h(\gamma) = \deg\gamma\cdot h(\gamma) = M(\gamma)
  \end{equation*}
  completing the proof.
\end{proof}

In our proof of Theorem \ref{RationalAttained}, it will be necessary to replace an arbitrary representation $\alpha = \alpha_1\cdots\alpha_N$ with 
another representation of $\alpha = \beta_1\cdots\beta_N$ that uses only rational numbers and satisfies
\begin{equation*}
  \sum_{n=1}^NM(\beta_n)^t \leq \sum_{n=1}^N M(\alpha_n)^t.
\end{equation*}
Our next lemma provides us with the necessary elementary number theoretic tools to do this.

\begin{lem} \label{NewSequence}
  Suppose that $m,r_1,\ldots,r_N$ are positive integers such that
  \begin{equation*}
    m\mid \prod_{n=1}^N r_n.
  \end{equation*}
  For $1\leq n\leq N$, recursively define the points $m_n$ by
  \begin{equation} \label{Recursive}
    m_1 = \gcd(r_1,m)\quad\mathrm{and}\quad m_n = \gcd\left(r_n,\frac{m}{\prod_{i=1}^{n-1}m_i}\right).
  \end{equation}
  Then we have that
  \begin{equation*}
    m = \prod_{n=1}^N m_n.
  \end{equation*}
\end{lem}

Before we provide the proof of Lemma \ref{NewSequence}, we make one clarification regarding the definition of $m_n$.
Naively, it would appear that
\begin{equation*}
  \frac{m}{\prod_{i=1}^{n-1}m_i}
\end{equation*}
is not necessarily an integer, so that taking its greatest common divisor with another integer might not be well-defined.  However, we note immediately that
$m_1\mid m$, which also implies that $m_2$ is well-defined.  Then clearly we have that $m_2 \mid m/m_1$ implying that $m_3$ is also well-defined.  
As we can see, it follows inductively that
\begin{equation*}
  m_n \mid \frac{m}{\prod_{i=1}^{n-1}m_i}
\end{equation*}
for all $1\leq n\leq N$, meaning, in particular, that $m_n$ is well-defined for all such $n$.  Now we may proceed with the proof of Lemma \ref{NewSequence}.

\begin{proof}[Proof of Lemma \ref{NewSequence}]
  We will assume that $m\ne \prod_{n=1}^Nm_n$ and find a contradiction.  Since the product $\prod_{n=1}^Nm_n$ divides $m$, there must exist a prime number $p$ for which
  \begin{equation} \label{NuMLower}
    \nu_p(m) > \sum_{j=1}^N\nu_p(m_j),
  \end{equation}
  where $\nu_p(x)$ denotes the highest power of $p$ dividing the integer $x$.   It now follows that
  \begin{equation*}
    \nu_p(m_n) < \nu_p(m) - \sum_{j=1}^{n-1}\nu_p(m_j) = \nu_p\left( \frac{m}{\prod_{j=1}^{n-1} m_j}\right)
  \end{equation*}
  for every $n\in\{1,\ldots,N\}$.  Hence, the definition of $m_n$ implies that
  \begin{equation*}
    \nu_p(m_n) = \min\left\{ \nu_p(r_n), \nu_p\left( \frac{m}{\prod_{j=1}^{n-1} m_j}\right)\right\} = \nu_p(r_n)
  \end{equation*}
  for every $n\in\{1,\ldots,N\}$.  It now follows from \eqref{NuMLower} that $\nu_p(m) > \sum_{n=1}^N\nu_p(r_n)$, contradicting our 
  assumption that $m$ divides $\prod_{n=1}^N r_n$.
\end{proof}

Now that we have established our key lemmas, we may now proceed with the proof of Theorem \ref{RationalAttained}.

\begin{proof}[Proof of Theorem \ref{RationalAttained}]
	As we have noted in the introduction, the case $t=\infty$ is known \cite{FiliSamuels}, so we proceed immediately to the situation where $1\leq t <\infty$.

  We may assume without loss of generality that $\alpha > 0$.  Since $\alpha$ is rational, there exist positive integers $m$ and $m'$ such that
  $\gcd(m,m') = 1$ and $\alpha = m/m'$.  Furthermore, by the results of \cite{SamuelstMetric1}, there exist surds $\alpha_1,\ldots,\alpha_N$ such that
  \begin{equation} \label{SurdEquality}
    \alpha = \alpha_1\cdots\alpha_N\quad\mathrm{and}\quad M_t(\alpha)^t = \sum_{n=1}^NM(\alpha_n)^t.
  \end{equation}
  Let $K$ be a number field containing $\alpha_1,\ldots\alpha_N$.  Now we may take the norm from $K$ to $\rat$ of both sides of the first equation in \eqref{SurdEquality}.
  We apply \eqref{NormDescend} and the fact that the $\norm_{K/\rat}$ is a homomorphism to establish that
  \begin{equation*}
    \left(\frac{m}{m'}\right)^{[K:\rat]} = \prod_{n=1}^N \norm_{K/\rat}(\alpha_n) = \prod_{n=1}^N\left(\norm_{\rat(\alpha_n)/\rat}(\alpha_n)\right)^{[K:\rat(\alpha_n)]}.
  \end{equation*}
  Suppose further that, for each $1\leq n\leq N$, $r_n$ and $s_n$ are relatively prime positive integers such that
  \begin{equation*}
    \frac{r_n}{s_n} = \pm \norm_{\rat(\alpha_n)/\rat}(\alpha_n).
  \end{equation*}
  Therefore, we have that
  \begin{equation*}
    \left(\frac{m}{m'}\right)^{[K:\rat]} = \pm\prod_{n=1}^N\left(\frac{r_n}{s_n}\right)^{[K:\rat(\alpha_n)]}.
  \end{equation*}
  It is obvious that $[K:\rat(\alpha_n)]\mid [K:\rat]$ so we obtain that
  \begin{equation*}
    m^{[K:\rat]}\mid \left(\prod_{n=1}^Nr_n\right)^{[K:\rat]}\quad\mathrm{and}\quad m'^{[K:\rat]}\mid \left(\prod_{n=1}^Ns_n\right)^{[K:\rat]}.
  \end{equation*}
  It follows from elementary number theory facts that
  \begin{equation}\label{Divides}
    m \mid \prod_{n=1}^Nr_n\quad\mathrm{and}\quad m'\mid \prod_{n=1}^Ns_n.
  \end{equation}

  Setting up the hypotheses of Lemma \ref{NewSequence}, we define recursive sequences corresponding to $m$ and $m'$.  First set
  \begin{equation*}
     m_1 = \gcd(r_1,m)\quad\mathrm{and}\quad m_n = \gcd\left(r_n,\frac{m}{\prod_{i=1}^{n-1}m_i}\right)
   \end{equation*}
   and
   \begin{equation*}
     m'_1 = \gcd(s_1,m')\quad\mathrm{and}\quad m'_n = \gcd\left(s_n,\frac{m'}{\prod_{i=1}^{n-1}m'_i}\right)
   \end{equation*}
   so we clearly have that
   \begin{equation} \label{RSInequalities}
     |r_n| \geq |m_n|\quad\mathrm{and}\quad |s_n|\geq |m'_n|.
   \end{equation}
   Applying Lemma \ref{NewSequence}, we have that
   \begin{equation*}
     m = \prod_{n=1}^Nm_n\quad\mathrm{and}\quad m' = \prod_{n=1}^Nm'_n
   \end{equation*}
   so that
   \begin{equation} \label{CorrectRationalRep}
     \alpha = \frac{m}{m'} = \prod_{n=1}^N\frac{m_n}{m'_n}.
   \end{equation}
   Now it follows from the definition of $M_t(\alpha)$ that
   \begin{equation} \label{CorrectUpper}
     M_t(\alpha)^t \leq \sum_{n=1}^NM\left(\frac{m_n}{m'_n}\right)^t,
   \end{equation}
   so we must show that the right hand side of \eqref{CorrectUpper} is also a lower bound for $M_t(\alpha)^t$.

   To see this, note that by Lemma \ref{NormConversion}, we have that
   \begin{equation*}
     M(\alpha_n) = M\left(\norm_{\rat(\alpha_n)/\rat}(\alpha_n)\right) = M\left(\frac{r_n}{s_n}\right)
   \end{equation*}
   for all $1\leq n\leq N$.  We have assumed that $r_n$ and $s_n$ are relatively prime, so it follows from known facts about the Mahler measure that
   \begin{equation*}
     M(\alpha_n) = \log \max\{|r_n|,|s_n|\}.
   \end{equation*}
   Then applying \eqref{RSInequalities}, we find that
   \begin{equation*}
     M(\alpha_n) \geq \log \max\{|m_n|,|m'_n|\} \geq M\left(\frac{m_n}{m'_n}\right),
   \end{equation*}
   and consequently,
   \begin{equation*}
     M_t(\alpha)^t = \sum_{n=1}^N M(\alpha_n)^t \geq \sum_{n=1}^N M\left(\frac{m_n}{m'_n}\right)^t.
   \end{equation*}
   Combining this with \eqref{CorrectRationalRep} and \eqref{CorrectUpper}, the result follows.
     
\end{proof}

\section{The quadratic case}

Our first lemma gives one particular set of points that attain the infimum in $M_t(\sqrt D)$ for all $t\in [1,\infty]$.  When $t>1$, we can also identify the Mahler measures 
of any points $\alpha_1,\ldots,\alpha_N$ attaining the infimum in $M_t(D^{1/k})$.

\begin{lem} \label{TameInfimum}
	Suppose that $p_1,\ldots,p_L$ are distinct primes written in decreasing order, $D=p_1\cdots p_L$, $t\in [1,\infty)$, and $k\in\nat$.
	The infimum in $M_t(D^{1/k})$ is attained by $p_1^{1/k},\ldots,p_L^{1/k}$ and
	\begin{equation*}
		M_t(D^{1/k})^t = \sum_{\ell = 1}^L(\log p_\ell)^t.
	\end{equation*}
	If $t>1$ and $\alpha_1,\cdots,\alpha_N$ are algebraic numbers attaining the infimum in $M_t(D^{1/k})$ then
	$N\geq L$.  Moreover, it is possible to relabel the elements $\alpha_1,\ldots,\alpha_N$ so that
	\begin{enumerate}[(i)]
		\item\label{Smalln} $M(\alpha_n) = \log p_n$ for all $n\leq L$, and
		\item\label{Largen} $M(\alpha_n) = 0$ for all $n > L$.
	\end{enumerate}
	In particular, $M(\alpha_n) \leq \log p_1$ for all $n$.
\end{lem}

\begin{proof}
We certainly have that $D^{1/k} = p_1^{1/k}\cdots p_\ell^{1/k}$, and by the definition of $M_t$, we know that
\begin{equation*} 
	M_t(D^{1/k})^t \leq \sum_{\ell =1}^L M(p_\ell^{1/k})^t.
\end{equation*}
For each $\ell$, we know that $x^k - p_\ell$ vanishes at $p_\ell^{1/k}$ and is irreducible by Eisenstein's criterion, so that $M(p_\ell^{1/k}) = M(x^k-p_\ell) = \log p_\ell$.
Hence, we find that
\begin{equation}\label{StandardUpper}
	M_t(D^{1/k})^t \leq \sum_{\ell = 1}^L (\log p_\ell)^t.
\end{equation}
To prove the first statement of the lemma, it is now sufficient to show that
\begin{equation} \label{StandardUpper2}
	M_t(D^{1/k})^t \geq \sum_{\ell = 1}^L (\log p_\ell)^t.
\end{equation}

Now suppose $\alpha_1,\ldots,\alpha_N\in\alg$ attain the infimum in $M_t(D^{1/k})$ and select a number field $K$ containing $D^{1/k},\alpha_1,\ldots,\alpha_N$.  
By definition, we know that $D^{1/k} = \alpha_1\cdots\alpha_N$.  Using the fact that  $\norm_{K/\rat}$ is a multiplicative homomorphism, we obtain that
\begin{equation*}
	\norm_{K/\rat}(D^{1/k}) = \prod_{n=1}^N\norm_{K/\rat}(\alpha_N)
\end{equation*}
so that
\begin{equation} \label{PreliminaryNorm}
	\left(\norm_{\rat(D^{1/k})/\rat}(D^{1/k})\right)^{[K:\rat(D^{1/k})]} = \prod_{n=1}^N\left(\norm_{\rat(\alpha_n)/\rat}(\alpha_N)\right)^{[K:\rat(\alpha_n)]}.
\end{equation}
Each of the above norms is a rational number.  Hence, for each $n$, there exist positive relatively prime integers $r_n$ and $s_n$ such that
\begin{equation*}
	|\norm_{\rat(\alpha_n)/\rat}(\alpha_n)| = \frac{r_n}{s_n}.
\end{equation*}
Again using Eisenstein's Criterion, we know that $x^k-D$ is the minimal polynomial of $D^{1/k}$ over $\rat$, implying that $|\norm_{\rat(D^{1/k})/\rat}(D^{1/k})| = D$.
Substituting these values into \eqref{PreliminaryNorm}, we find that
\begin{equation} \label{DPower}
	D^{[K:\rat(D^{1/k})]} = \prod_{n=1}^N\left(\frac{r_n}{s_n}\right)^{[K:\rat(\alpha_n)]}.
\end{equation}
For each $n$, $\alpha_n$ has minimal polynomial of the form
\begin{equation*}
	\hat f_n(x) = x^d + \frac{a_{d-1}}{b_{d-1}}x^{d-1} + \cdots + \frac{a_1}{b_1}x \pm \frac{r_n}{s_n}
\end{equation*}
over $\rat$ for integers $a_1,\ldots,a_{d-1},b_1,\ldots,b_{d-1}$ with $b_i\ne 0$ and $(a_i,b_i) = 1$.  Hence, its minimal polynomial over $\intg$ is given by
\begin{equation*}
	f_n(x) = \lcm(s_n, b_{d-1}, \ldots, b_1)\cdot x^d + \cdots \pm r_n \cdot \lcm(s_n, b_{d-1},\ldots, b_1)
\end{equation*}
and its Mahler measure satisfies
\begin{equation*} \label{RLower}
	M(\alpha_n) \geq \log \left(\frac{r_n}{s_n}\cdot \lcm(s_n, b_{d-1},\ldots, b_1)\right) \geq \log r_n.
\end{equation*}

For each $n$, let
\begin{equation*}
	P_n = \left\{ p\in\{p_1,\ldots,p_L\} : p\mid r_n\right\}.
\end{equation*}
We have assumed that $\alpha_1,\ldots,\alpha_N$ attains the infimum in $M_t(D^{1/k})$, so we get that
\begin{equation} \label{SlightlyFancyLower}
	M_t(D^{1/k})^t = \sum_{n=1}^N M(\alpha_n)^t \geq \sum_{n=1}^N (\log r_n)^t \geq \sum_{n=1}^N\left( \sum_{p\in P_n} \log p\right)^t.
\end{equation}
Since $t \geq 1$, we always have that
\begin{equation} \label{Mink}
	\left( \sum_{p\in P_n} \log p\right)^t \geq \sum_{p\in P_n} (\log p)^t,
\end{equation}
which implies that
\begin{equation*}
	M_t(D^{1/k})^t \geq \sum_{n=1}^N\sum_{p\in P_n} (\log p)^t.
\end{equation*}
However, applying \eqref{DPower}, we know that for each $\ell\in\{1,\ldots,L\}$, there exists $n\in\{1,\ldots,N\}$ such that $p_\ell\in P_n$, establishing
\eqref{StandardUpper2} and the first statement of the lemma.

Now assume that $t>1$.   If $|P_n|\geq 2$, then we must have strict inequality in \eqref{Mink}.
Therefore, if $|P_n| \geq 2$ for some $n$, then \eqref{SlightlyFancyLower} implies that
\begin{equation*}
	M_t(D^{1/k})^t > \sum_{n=1}^N \sum_{p\in P_n} (\log p)^t \geq \sum_{\ell = 1}^L (\log p_\ell)^t
\end{equation*}
contradicting \eqref{StandardUpper}.  Therefore, $|P_n| \leq 1$ for every $n$ and we have established that
\begin{enumerate}[(a)]
	\item For every $\ell$, there exists $n$ such that $p_\ell\mid r_n$, and
	\item If $\ell_1 \ne \ell_2$ then we can never have that $p_{\ell_1}\mid r_n$ and $p_{\ell 2}\mid r_n$.
\end{enumerate}
It follows from the box principle that $N\geq L$.  Moreover, we may reorder $\alpha_1,\ldots,\alpha_N$ such that $p_n\mid r_n$
for all $1\leq n\leq L$, which shows that
\begin{equation} \label{EqualityLower}
	M(\alpha_n) \geq \log r_n \geq \log p_n\quad\mathrm{for}\ 1\leq n\leq L.
\end{equation}
If we have strict inequality in \eqref{EqualityLower} for some $n$, then
\begin{equation} \label{Strictness}
	M_t(D^{1/k})^t =\sum_{n=1}^N M(\alpha_n)^t > \sum_{\ell = 1}^L (\log p_\ell)^t
\end{equation}
contradicting \eqref{StandardUpper} and establishing \eqref{Smalln}.  Similarly, if $M(\alpha_n) > 0$ for some $n>L$, then
\eqref{Strictness} holds as well verifying \eqref{Largen}.

\end{proof}

Now that we have proven Lemma \ref{TameInfimum}, the proof of Theorem \ref{QuadraticValue} is essentially complete.  Indeed, when $t\in [1,\infty)$
Theorem \ref{QuadraticValue} is simply the first statement of Lemma \ref{TameInfimum}, and the case $t=\infty$ was given already in \cite{FiliSamuels}.
The only task remaining is to prove Theorem \ref{QuadraticNotAttained}, in which the second statement of Lemma \ref{TameInfimum} plays a key role.

Before proceeding, we establish some conventions that will be used for the remainder of this article.  For $d\in\intg$ and $r\in\rat$, we say that {\it $d$ divides $r$} if when
$r$ is written $r=m/n$ with $m \in \nat$, $n \in\intg\setminus\{0\}$ and $(m, n)=1$, then either $d\mid m$ or $d\mid n$.   We say that {\it $d$ divides the numerator or denominator of $r$}
if $d$ divides $m$ or $n$, respectively.

We say that an algebraic number $\alpha$ is {\it stable} if all of its conjugates lie either inside the open unit disk, on the unit circle, or outside the closed unit disk.
Otherwise, we say that $\alpha$ is {\it unstable}.  It is clear that all rational numbers and all imaginary quadratic numbers are stable, while real quadratic numbers can 
be either stable or unstable.   If $\alpha$ is any algebraic number having minimal polynomial
\begin{equation*}
	f(x) = a_Nx^{N} + \dots + a_1x + a_0,
\end{equation*}
then it is simple to verify that
\begin{equation*}
	M(\alpha) \geq \log\max\{|a_N|,|a_0|\}
\end{equation*}
with equality if and only if $\alpha$ is stable.  We now state a simple criterion which allows us to determine if a quadratic algebraic number is stable by considering 
the coefficients of the minimal polynomial.

 \begin{lem} \label{quadr} 
 Suppose that $\alpha$ is a quadratic algebraic number having minimal polynomial $f(x) = ax^2 + bx + c$ over $\intg$.  We have that $\alpha$ is stable if and only if
 $|a+c| > |b|$.  In this situation, the following hold.
 \begin{enumerate}[(i)]
 \item\label{outside} If $|a|<|c|$ then both conjugates of $\alpha$ have modulus greater than one.
 \item\label{on} If $|a|=|c|$ then both conjugates of $\alpha$ have modulus one.
 \item\label{inside} If $|a| > |c|$ then both conjugates of $\alpha$ have modulus less than one.
 \end{enumerate}
 \end{lem}
 
 \begin{proof} 
   Suppose that $f(x) = a(x-\alpha)(x-\beta)$.  If $f(1)$ and $f(-1)$ have opposite signs, then $f$ has precisely one root in the interval $(-1,1)$.  
   The other root must also be real and lie outside of $(-1,1)$, so $\alpha$ is unstable.  If $f(1)$ and $f(-1)$ have the same sign, then $f$ has either zero or two roots in $(-1,1)$.
   In the case of two roots in $(-1,1)$, $\alpha$ is clearly stable.  If $f$ has zero roots in $(-1,1)$, then it either has two complex roots, in which case $\alpha$ is certainly stable, or
   two real roots both lying outside of $[-1,1]$, also implying that $\alpha$ is stable.

   We have now shown that $\alpha$ is stable if and only if $f(1) = a+b+c$ and $f(-1) = a-b+c$ have the same sign.
   Clearly, $f(1)$ and $f(-1)$ are both positive if and only if $a+c > |b|$ and both negative if and only if $-a-c > |b|$.  Thus, $\alpha$ is stable if and only if $|a+c|>|b|$.
   
   If, in addition, $|a| < |c|$, then $|\alpha\beta| = |c|/|a| > 1$, so both $\alpha$ and $\beta$ have modulus greather than $1$.  Similarly, if $|a| > |c|$, then
   $|\alpha\beta| = |c|/|a| < 1$ implying that both $\alpha$ and $\beta$ have modulus less than $1$.  Finally, if $|a| = |c|$ then $|\alpha\beta| = 1$.  Since $\alpha$
   is stable, $\alpha$ and $\beta$ must be complex conjugate numbers both of modulus $1$.
\end{proof}
 
 The following lemma shows us that certain quadratic algebraic numbers, which we will encounter in the proof of Theorem \ref{QuadraticNotAttained}, have
 relatively simple minimal polynomials.

\begin{lem}\label{small}
Let $D$ be a square-free integer, $p$ be a prime divisor of $D$, and $\alpha$ a quadratic algebraic number in $\rat(\sqrt D)$.  If $M(\alpha) \leq \log p$ and $p$ divides 
the numerator of $\norm_{\rat(\sqrt D)/\rat}(\alpha)$ then $\alpha$ is stable.  Moreover, the minimal polynomial of $\alpha$ satisfies
\begin{equation*}
	f(x) = ax^2\pm p\quad\mathrm{or}\quad f(x) = ax^2\pm px +p
\end{equation*}
where $a$ is a positive integer with $a<p$.
\end{lem}

\begin{proof}
Suppose that $f(x) = ax^2+bx+c \in \Z[x]$ is the minimal polynomial of $\alpha$ over $\intg$, so we may assume that $a>0$.  Since $\alpha$ has degree $2$, we have that
\begin{equation*}
	\norm_{\rat(\sqrt D)/\rat}(\alpha) = \alpha\bar\alpha = \frac{c}{a}
\end{equation*}
where $\bar\alpha$ is the conjugate of $\alpha$ over $\rat$.  We have assumed that $p$ divides the numerator of $\norm_{\rat(\sqrt D)/\rat}(\alpha)$, which itself must divide $c$,
implying that $p\mid c$.  Since $M(\alpha) \geq \log\max\{|a|,|c|\}$, we have that 
\begin{equation*}
	\log p \leq \log |c| \leq \log\max\{|a|,|c|\} \leq M(\alpha) \leq \log p,
\end{equation*}
and we conclude that 
\begin{equation} \label{CEqual}
	M(\alpha) = \log |c| = \log p.
\end{equation}
It now follows that $|a| \leq |c|$ and, since $M(\alpha)$ is the log of an integer, we further obtain that $\alpha$ is stable.  Hence, Lemma \ref{quadr} implies that $|a+c| > |b|$.

We cannot have $|a| = |c|$, since 
\begin{equation*}
	\norm_{\rat(\sqrt D)/\rat}(\alpha) = \frac{c}{a} = \pm 1
\end{equation*}
is not divisible by $p$, so it follows that $|a| < |c|$. In view of Lemma \ref{quadr} \eqref{outside}, we have that $|\alpha|,|\bar\alpha| > 1$.  Therefore, we find that
\begin{equation} \label{bBound}
	|b| < |a+c| \leq |a| + |c| < 2|c| = 2p.
\end{equation}
Now let $\Delta = b^2-4ac$. Since $\rat(\sqrt{\Delta}) = \rat(\sqrt{D})$, and $D$ is square-free, we have 
$\Delta = Dv^2$ for some $v \in \intg$. The quadratic formula gives 
\begin{equation*}
	\norm_{\rat(\sqrt D)/\rat}(\alpha) = \alpha\bar\alpha = \frac{b^2-Dv^2}{4a^2},
\end{equation*}
and since $p\mid D$ and the numerator of $\norm_{\rat(\sqrt D)/\rat}(\alpha)$, it follows that $p$ divides $b^2$.   Of course, this implies that $p \mid b$.
Using \eqref{bBound}, we now see that $b\in\{0,p,-p\}$.

If $b = 0$ then we have by \eqref{CEqual} that $f (x)=ax^2\pm p$, establishing the lemma in this case. If $b= \pm p$, then $|a+c|>|b|$ holds if and only if $a$ and $c$ has the same sign. 
So in this situation, \eqref{CEqual} yields that $c=p$ which leads to $f(x) =ax^2 \pm px+p$.
\end{proof}

\begin{proof}[Proof of Theorem \ref{QuadraticNotAttained}]
By Theorem \ref{QuadraticValue}, we know that
\begin{equation*}
	M_t(\sqrt D) = \begin{cases} \left(\sum_{\ell = 1}^L (\log p_\ell)^t\right)^{1/t} & \mathrm{if\ } t\in (1,\infty) \\
					\log p_1 & \mathrm{if\ } t = \infty. \end{cases}
\end{equation*}
We also observe that
\begin{equation} \label{DRep}
	\sqrt D = \sqrt{\frac{p_1}{p_2\cdots p_L}} \cdot p_2\cdots p_L
\end{equation}
and that each term in the product on the right hand side of \eqref{DRep} belongs to $\rat(\sqrt D)$.  We obviously have that $M(p_\ell) = \log p_\ell$ for all $\ell$.
Furthermore, our assumption that $D < p_1^2$ ensures that $p_2\cdots p_L < p_1$, so it follows that
\begin{equation*}
	M\left( \sqrt{\frac{p_1}{p_2\cdots p_L}}\right) = \log p_1.
\end{equation*}
Combining these observations, we see that
\begin{equation*}
	 M\left( \sqrt{\frac{p_1}{p_2\cdots p_L}}\right)^t + \sum_{\ell = 2}^L M(p_\ell)^t =  \sum_{\ell = 1}^L (\log p_\ell)^t = M_t(\sqrt D)^t 
\end{equation*}
when $1<t<\infty$ and
\begin{equation*}
	\max\left\{ M\left( \sqrt{\frac{p_1}{p_2\cdots p_L}}\right), M(p_2),\ldots, M(p_L)\right\} = \log p_1 = M_\infty(\sqrt D).
\end{equation*}
establishing one direction of the theorem as well as the second statement.

To prove the other direction, we assume that there exist points $\alpha_1,\ldots,\alpha_N\in\rat(\sqrt D)$ that attain the infimum in $M_t(\sqrt D)$,
and for simplicity, we set $p=p_1$.  When $t\in (1,\infty)$, Lemma \ref{TameInfimum} establishes that that $M(\alpha_n) \leq \log p$ for all $n$.
In the case $t=\infty$, we also have $M(\alpha_n) \leq \log p$ for all $n$ as a consequence of Theorem \ref{QuadraticValue}.
Since $\sqrt D = \alpha_1\cdots\alpha_N$, we have that
\begin{equation}\label{norms}
-D = \norm_{\rat(\sqrt D)/\rat}(\sqrt D) = \prod_{n=1}^N\norm_{\rat(\sqrt D)/\rat}(\alpha_n).
\end{equation}
Defining the set
\begin{equation*}
	\Lambda = \left\{1\leq n\leq N: p\mid \norm_{\rat(\sqrt D)/\rat}(\alpha_n)\right\}
\end{equation*}
we apply \eqref{norms} to see that
\begin{equation} \label{Pnorms}
	\sum_{n\in\Lambda}\nu_{p}\left(\norm_{\rat(\sqrt D)/\rat}(\alpha_n)\right) = \nu_{p}(D) = 1,
\end{equation}
where the last equality follows since $D$ is square-free.  If $\Lambda$ contains no irrational points, then we have that
\begin{equation*}
	p\mid \norm_{\rat(\sqrt D)/\rat}(\alpha_n) = \alpha_n^2
\end{equation*}
for all $n\in\Lambda$.  However, this implies that $\nu_p(\norm_{\rat(\sqrt D)/\rat}(\alpha_n))$ is even for all $n\in\Lambda$.  It follows that
the left hand side of \eqref{Pnorms} is also even, a contradiction.

We have shown that there must exist $n$ such that $\alpha_n$ is quadratic, $M(\alpha_n)\leq \log p$, and $p$ divides $\norm_{\rat(\sqrt D)/\rat}(\alpha_n)$.
If $p$ divides the numerator of $\norm_{\rat(\sqrt D)/\rat}(\alpha_n)$, then we may apply Lemma \ref{small} to see that $\alpha_n$ is stable and is a root of
\begin{equation*}
	f(x) = ax^2\pm p\quad\mathrm{or}\quad f(x) = ax^2\pm px +p
\end{equation*}
for some positive integer $a<p$.

Suppose now that $\Delta$ is the discriminant of $f$.  Since $\alpha_n$ is quadratic over $\rat$, we have $\rat(\sqrt{\Delta}) = \rat(\sqrt{D})$. 
Furthermore, since $D$ is a square-free, we have that $\Delta=Dv^2$ for some $v \in \nat$.  If $f(x) = ax^2\pm p$, we see that $\Delta = \pm 4ap$,
so that 
\begin{equation*}
	pp_2\cdots p_L v^2 = Dv^2 = \pm 4ap.
\end{equation*}
Since $p_2,\ldots,p_L$ are distinct primes,  we obtain that $p_2\cdots p_L \mid a$, and hence,
\begin{equation*}
	p_2\cdots p_L \leq a <p,
\end{equation*} 
establishing that $D < p^2$ in this case.

If $f(x) = ax^2\pm px +p$ then $\Delta = p^2-4ap=p(p-4a)$. We have assume that $D$ is positive so that $p - 4a > 0$, and, trivially, $p-4a < p$. 
Hence, $D \leq \Delta = p(p-4a) < p^2$ completing the proof when $p$ divides the numerator of $\norm_{\rat(\sqrt D)/\rat}(\alpha_n)$.

If $p$ divides the denominator of $\norm_{\rat(\sqrt D)/\rat}(\alpha_n)$ instead, the $p$ must divide the numerator of $\norm_{\rat(\sqrt D)/\rat}(\alpha_n^{-1})$.
Of course, we also have that $M(\alpha_n^{-1})\leq \log p$ and $\alpha_n^{-1}\in\rat(\sqrt D)$ is quadratic, so we may apply the above argument to $\alpha_n^{-1}$
in place of $\alpha_n$.
\end{proof}

\end{document}